\documentclass[11pt]{article}
\usepackage{geometry}
\usepackage[utf8]{inputenc}
\usepackage{amsmath,amsthm,amssymb,color,graphicx,url,hyperref,diagbox,enumerate,cite,tikz,authblk}
\usepackage{appendix}
\usepackage{makecell}

\newtheorem{theorem}{Theorem}

\newtheorem{problem}[theorem]{Problem}

\newtheorem{lemma}[theorem]{Lemma}

\theoremstyle{definition}
\newtheorem{remark}[theorem]{Remark}

\newcommand{\N}{\mathbb{Z}^+}

\newcommand{\R}{\mathbb{R}}

\newcommand{\E}{\mathcal{E}}

\newcommand{\T}{\mathcal{T}}

\newcommand{\tr}{\operatorname{tr}}

\usepackage{graphicx}

\title{An elementary proof of the 2-regularity of Pythagorean triples}
\author{William J. Wesley\thanks{Discrete Mathematics Group, Institute for Basic Science (IBS), Daejeon,
South Korea.\\ This work was supported by the Institute for Basic Science (IBS-R029-C1).}}
\date{\today}

\begin{document}

\maketitle

\section{Introduction}

In the 1980s, Ron Graham offered a \$100 prize to determine whether every 2-coloring of the positive integers must contain a monochromatic Pythagorean triple, that is, a solution to the \emph{Pythagorean equation} $x^2 + y^2 = z^2$ \cite{NaturePythagTriples}. The problem was solved in the affirmative by Heule, Kullmann, and Marek in 2016 \cite{PythagoreanTriplesSAT}, and in fact they proved the stronger result that every 2-coloring of $\{1,...,7825\}$ contains a monochromatic triple, but there exists a coloring of $\{1,...,7824\}$ that does not. Their proof relied on a large Boolean satisfiability (SAT) solving computation that produced a proof certificate nearly 200 terabytes in size. SAT solving has proven effective in determining exact values for other constants in arithmetic Ramsey theory such as Schur, van der Waerden, and Rado numbers \cite{SchurFive,WJW_Rado_ISSAC,BMRS_3ColorSchur,AhmedZamanBright,VDW26,VDW34}, but these results can take significant computational investment to reproduce or verify.  

More generally, an equation $\E$ is \emph{$k$-regular} if every $k$-coloring of the positive integers contains a monochromatic solution to $\E$ and \emph{regular} if it is $k$-regular for all $k$. Some authors use the term \emph{partition regular} instead of \emph{regular}. For linear homogeneous equations $\E$,  Rado's theorem \cite{RadoThesis} gives precise conditions for when $\E$ is regular, but the regularity of higher degree polynomial equations and other types of equations is not understood as well. It is still open whether the Pythagorean equation is 3-regular, and determining whether it is regular is a challenging open problem which was also asked by Graham \cite{GrahamRamseyOpenProblems}. Frantzikinakis, Klurman, and Moreira recently made progress by showing that Pythagorean \emph{pairs} are regular, meaning for every $k$-coloring of the positive integers, there are two distinct integers $x,y$ such that $x$ and $y$ appear in the same Pythagorean triple and $x$ and $y$ share the same color  \cite{PythagPairsPartitionRegular}. Another related result by Chow, Lindqvist, and Prendiville shows that the equation $x_1^2 + x_2^2 + x_3^2+x_4^2 = x_5^2$ is regular and gives a generalization of Rado's theorem for sums of $k$-th powers \cite{ChowLindqvistPrendiville}. Both these results require lengthy and non-elementary analytic arguments. A more detailed survey of the regularity of homogeneous quadratic equations and related problems is given in \cite{frantzikinakis2025partitionregularityhomogeneousquadratics}.

The purpose of this note is to answer Graham's original question with an elementary proof that is free of heavy computation. Specifically, we will give a new proof of the following statement. 

\begin{theorem}\label{ThmMain}
    There exists a positive integer $N$ such that every $2$-coloring of $\{1,\dots,N\}$ contains a monochromatic solution to $x^2+ y^2 = z^2$.
\end{theorem}

Of course, this result is weaker than the result in \cite{PythagoreanTriplesSAT}. The value of $N$ that emerges from our proof is astronomically larger than the optimal value 7825, and we do not expect that our method can produce an $N$ anywhere close to the optimum.  

Our proof combines several existing ideas in a new way. The one computational step of the proof requires verification of properties of a particular matrix $A$. The matrix $A$ was found by semidefinite programming techniques, which have been used in computational arithmetic Ramsey theory to minimize the asymptotic number of monochromatic solutions to certain linear equations in \cite{ParriloRobertsonSaracino,DeLoeraVenturaWangWesley,Robertson2DSchur}. Our setup here is slightly different, and the verification can be done quickly (in under a minute) on a laptop. The theoretical ideas in the proof draw from other work on equation regularity. In particular, we implicitly use the notion of F{\o}lner sets that have been studied in, for instance, \cite{PythagPairsPartitionRegular,NewResultsMultAddRamsey}. However, our presentation of the proof is self-contained and requires only linear algebra, elementary number theory, and standard counting methods. The main part of the proof is given in Section \ref{SectionMainProof}, and the computational details are discussed in Section \ref{Section_A_computations}. All computational data for this work is available at \cite{WJWPythagGithub}. 

\begin{remark}
    Another preprint giving a new proof of Theorem \ref{ThmMain} was posted by Leng days before this one \cite{LengPythagTriples}. The proof there has a different flavor and uses analytic and ergodic theoretic tools. This work was carried out independently, without knowledge of \cite{LengPythagTriples}.    
\end{remark}

\section{Proof}\label{SectionMainProof}

We let $\chi : \N \to \{-1,1\}$ be a 2-coloring of the positive integers and suppose there are no monochromatic Pythagorean triples. It follows that for all positive integers $k$ and Pythagorean triples $(a,b,c)$ that 
\begin{equation}
    \chi(ka)\chi(kb) + \chi(ka)\chi(kc) + \chi(kb)\chi(kc) = -1.
\end{equation}

Let $\T$ be a particular collection of 68 primitive Pythagorean triples generated by the following steps. First, take the set of unordered Pythagorean triples $$U :=\{\{m^2-n^2,2mn,m^2+n^2\}: m-n \text{ odd},\ \gcd(m,n) =1,\  1 \le n < m \le 200\}.$$ For each triple in $U$, write the elements in increasing order $a < b < c$, and add the triple $(a,b,c)$ to $\T$ if the largest prime factor of $abc$ is at most 41. Let $$S = \{ dr: r \text{ belongs to some triple in } \T, 1\le d \le 8 \}.$$ We have $|S| = 1068$, and write the entries of $S$ in increasing order $s_1 < s_2 < \dots < s_{1068}$, and set $g_{ij} = \gcd(s_i,s_j)$. Let $P$ be the set of primes that are prime factors of some integer in $S$, and by the construction of $\T$, we see that $P$ is the set of primes that are at most 41.   

The key to the proof is the existence of a particular $|S| \times |S|$ matrix $A$ with nice properties. The entries of $A$, along with the sets $\T$ and $S$, are available at \cite{WJWPythagGithub}. For coprime positive integers $u,v$, let 

$$W(u,v) = \sum_{i < j, \\s_i/g_{ij} =  u,\\s_j/g_{ij} = v} A_{ij}.$$ Each row $i$ and column $j$ in $A$ corresponds to the elements $s_i$ and $s_j$ in $S$. The value $W(u,v)$ is the total weight of all entries of $A$ where $s_i$ and $s_j$ are in the same ratio as $u$ and $v$, i.e., $s_j/s_i = v/u$. Let $$Q = \{(u,v): u < v, s_i/g_{ij} = u, s_{j}/g_{ij} = v \text{ for some }s_i,s_j \in S\}.$$ The matrix $A$ is constructed so that for all $(u,v) \in Q$, if $u$ and $v$ do not appear together in any triple in $\T$, then $W(u,v) = 0$. Otherwise, they do appear together in some triple $t = (a,b,c)$, and $W(a,b) = W(a,c) = W(b,c)$. We call this common weight $K_t$.  

\begin{lemma}\label{Lemma_A_properties}
    The matrix $A$ satisfies the following.
    \begin{enumerate}[(i)]
        \item $A$ is positive semidefinite.
        \item For any triple $t = (a,b,c) \in \T$, we have $W(a,b) = W(a,c) = W(b,c) =: K_t$.
        \item If $u$ and $v$ do not appear together in any triple in $\T$, then $W(u,v) = 0$.
        \item $\tau:=2\sum_{t\in \T}K_t - \tr(A) > 0.$
    \end{enumerate}
\end{lemma}

Verifying the four properties in Lemma \ref{Lemma_A_properties} is the only computational step in the proof. This is not difficult, but we will defer the proof of Lemma \ref{Lemma_A_properties} to Section \ref{Section_A_computations}. 

For a fixed positive integer $L$, let $F_L = \{ \prod_{p \in P} p^{e_p} : 0 \le e_p < L\}$. We will set the value of $L$ later. Moreover, let $$C(u,v) = \frac{1}{|F_L|}\sum_{k \in F_L } \chi(ku)\chi(kv).$$

A crucial intermediate result that we need is that $C(u,v)$ does not change too much when $u$ and $v$ are dilated by a factor $g$.
\begin{lemma}\label{LemmaCbound}
    For a positive integer $g$ whose prime factors are all in $P$, let $\Omega(g) = \sum_{p \in P} v_p(g)$, where $v_p(g)$ is the exponent of the largest power of $p$ that divides $g$ (that is, $v_p(g)$ is the $p$-adic valuation of $g$). 

    We have $$|C(u,v)-C(gu,gv)| \le \frac{2\Omega(g)}{L}.$$
\end{lemma}

\begin{proof}
 Observe that $|F_L| = L^{|P|}$ since there are $L$ choices of exponent $e_p$ for each $p \in P$. If $x = \displaystyle\prod_{p \in P} p^{e_p}\in F_L \triangle gF_L$, where $\triangle$ denotes the symmetric difference, then there is some prime $p \in P$ such that exactly one of the conditions $0 \le e_p < L$ and $v_p(g) \le e_p < L+v_p(g)$ is violated. In this case, we say that $x$ is \emph{$p$-bad}. For a given prime $p$, there are at most $2v_p(g)L^{|P|-1}$ elements of $F_L \triangle gF_L$ that are $p$-bad, and summing over all primes $p\in P$ gives $\frac{|F_L \triangle gF_L|}{|F_L|} \le \frac{2\Omega(g)}{L}$. 

 

 Then we have $$|C(u,v) - C(gu,gv)| \le \frac{1}{|F_L|}\sum_{k \in F_L \triangle gF_L} |\chi(ku)\chi(kv)| \le \frac{2\Omega(g)}{L}$$
 because each term $\chi(ku)\chi(kv)$ in $C(u,v)$ has absolute value 1.
\end{proof}

Let $x(k) = (\chi(ks_1),\dots, \chi(ks_{1068}))^T$. Since $A$ is positive semidefinite, we have 
 \begin{equation} \label{EqE_nonneg}
     E:=\frac{1}{|F_L|}\sum_{k \in F_L} x(k)^TAx(k) \ge 0. 
 \end{equation}

 Expanding the quadratic form and using the fact that $\chi(ku)^2 = 1$, we have 

 \begin{align*}
     E &= \frac{1}{|F_L|}\sum_{k \in F_L}\left[\tr(A) + 2\sum_{i < j}A_{ij}\chi(ks_i)\chi(ks_j) \right] \\
     &= \tr(A) + \sum_{i < j}2A_{ij}C(s_i,s_j).
 \end{align*}

 By Lemma \ref{LemmaCbound}, we can normalize the arguments of $C$ without changing the value of $E$ too much, and we have 

\begin{align*}
    E \le \tr(A) + \sum_{i < j} 2A_{ij} C(\frac{s_i}{g_{ij}},\frac{s_j}{g_{ij}}) + \frac{4}{L}\sum_{i < j} |A_{ij}|\Omega(g_{ij}).
\end{align*}

Let $H$ be the constant value $\sum_{i < j} |A_{ij}|\Omega(g_{ij})$. We will rearrange the first sum by grouping terms according to the arguments of $C$. We have 

\begin{align*}
    \sum_{i < j} 2A_{ij} C(\frac{s_i}{g_{ij}},\frac{s_j}{g_{ij}}) &= 2\sum_{(u,v) \in Q}C(u,v)\sum_{i < j, s_i/g_{ij} = u,s_j/g_{ij} = v} A_{ij}\\ &= 2\sum_{(u,v) \in Q}C(u,v)W(u,v)
\end{align*}
Lemma \ref{Lemma_A_properties} states that the only $(u,v)$ with nonzero $W(u,v)$ are those that appear in some triple $t$ of $\T$. Each pair of distinct elements of $t$ contributes the same value, $K_t$. Therefore we have

\begin{align}
    E &\le \frac{4H}{L} + \tr(A) + 2\sum_{t=(a,b,c)\in \T} K_t[C(a,b)+C(a,c)+C(b,c)] \nonumber\\ 
    &= \frac{4H}{L} + \tr(A) + 2\sum_{t=(a,b,c)\in \T} K_t\Bigg[\frac{1}{|F_L|}\sum_{k \in F_L}(\chi(ka)\chi(kb)+\chi(ka)\chi(kc)+\chi(kb)\chi(kc))\Bigg] \nonumber\\ 
    &= \frac{4H}{L} + \tr(A) - 2\sum_{t \in \T} K_t \nonumber\\
    &= \frac{4H}{L} - \tau. \label{EqE_ineq}
\end{align}

By Lemma \ref{Lemma_A_properties}, we have $\tau > 0$, and so for sufficiently large $L$ we have $E < 0$. This contradicts \eqref{EqE_nonneg}, so $\chi$ must have a monochromatic solution. We can get an explicit $N$ by carrying out a few more calculations. We have $\tau = 728000000$ and $H = 23071970671176$, so we may take $L = 126770$, the smallest integer that makes the bound in \eqref{EqE_ineq} negative. 
We may take $N = \max(S)\max(F_L)$, the largest integer colored in the sum defining $E$. This value is $$N = 336200 (\prod_{p \in P}p)^{L-1} = 336200\cdot 304250263527210^{126769},$$
and the proof is complete except for the verification of Lemma \ref{Lemma_A_properties}. 

\section{Computational details}\label{Section_A_computations}

This section discusses some of the computational considerations of the proof, and we begin by commenting on the proof of Lemma \ref{Lemma_A_properties}. Parts $(ii)-(iv)$ are straightforward programming exercises, noting we can compute the values of $W(u,v)$ by iterating over the entries of $A$ once. Computer algebra systems such as Sage can check directly that $A$ is positive semidefinite, but this relies on floating point calculations, which can introduce errors. To increase confidence in our computations, we supply integer matrices $B$ and $R$ such that \begin{equation}\label{EqBR_def}BB^T + R = 10^8 A.\end{equation}
Both matrices are available at \cite{WJWPythagGithub}. The matrix $R$ satisfies
\begin{equation}\label{EqR}R_{ii} \ge \sum_{j \neq i} |R_{ij}|
\end{equation}
for all $i$, and this is also simple to verify by computer. One can deduce that $R$, and therefore $A$, is positive semidefinite via the Gershgorin circle theorem and the nonnegative eigenvalue characterization of positive semidefinite matrices. We also give a short, self-contained proof below.

\begin{proof}[Proof of Lemma \ref{Lemma_A_properties} ($i$)]
Observe that $R$ is symmetric because $BB^T$ and $A$ are symmetric. Then by \eqref{EqBR_def}, \eqref{EqR}, and careful grouping of terms, we have for any $x \in \R^{|S|}$ that 

\begin{align*}
x^TAx &= 10^{-8}\bigg(\|B^Tx\|^2 + \sum_{i,j} R_{ij} x_ix_j\bigg) \\
&\ge  10^{-8}\bigg(\|B^Tx\|^2 + \sum_{i}\sum_{j \neq i} |R_{ij}| x_i^2+\sum_{i\neq j} R_{ij} x_ix_j\bigg) \\
&= 10^{-8}\bigg(\|B^Tx\|^2 + \sum_{i < j} |R_{ij}|(x_i +\operatorname{sgn}(R_{ij})x_j)^2\bigg) \\
& \ge 0.
\end{align*}
\end{proof} 

This completes the proof of Theorem \ref{ThmMain}. However, the reader may complain that the matrices $A, B$, and $R$ appear to arise out of nowhere, and we believe it is worthwhile to explain how to find them. 

A suitable matrix $A$ was found via the following semidefinite program, with $\T$ and $S$ the fixed sets in Section \ref{SectionMainProof}.

\begin{align*}
    \min \tr(A)  \quad \text{subject to} \\
    W(a,b) = W(a,c) = W(b,c) = K_t,&&\text{ for all $ t =(a,b,c) \in \T$}, \\ 
    W(u,v) = 0, &&\text{ if $u,v$ do not appear together in any triple in $\T$}, \\
    2\sum_{t\in \T}K_t =1, \\ 
    A \text{ positive semidefinite}.
\end{align*}

A feasible solution to this program satisfies  properties ($i$)-($iii$) in Lemma \ref{Lemma_A_properties} by design. To satisfy property ($iv$), we only need a feasible solution $A$ to have $\tr(A) < 1$, and we do not need to find an optimal solution. We also rescaled $A$ so that it has integer entries, and this rescaled matrix still has all of the properties of Lemma \ref{Lemma_A_properties}. 

To find $B$, we run a Cholesky factorization algorithm on the matrix $A - 10^5 I$ and find a matrix $L$ such that 
$$LL^T \approx A -10^5 I.$$
This can be done quickly (in less than  a second), but we remark that some implementations are slow when the entries of $A-10^5I$ are treated as integers rather than floating-point numbers, so some care should be taken. The $\approx$ symbol is used because this a numerical approximation. Then we let $B$ be the matrix obtained by rounding the entries of $10^4L$ to the nearest integer and set $R = 10^8 A - BB^T$. Note that we factored $A - 10^5I$ rather than $A$ because after rounding, we want ample buffer for $R$ to satisfy \eqref{EqR}.

\section*{Acknowledgments}

We acknowledge that the outline of the proof in this article was found by an advanced LLM after the author prompted it for a non-computational proof of Theorem \ref{ThmMain}. The author reviewed the output, checked its correctness, and then refined the argument. He also added additional context and references that were omitted by the LLM. The author used an LLM only for proofreading and minor edits during the writing process. All computational aspects of the proof were independently verified by the author in Sage. The author takes and claims full responsibility for this result and preparation of this work.   

\bibliography{RamseyCayley}
\bibliographystyle{abbrv}

\end{document}